\documentclass[10pt,a4paper]{amsart}
\usepackage{amssymb,amsmath,amsfonts}

\headsep=1.2500cm \numberwithin{equation}{section}
\newtheorem{Theorem}{Theorem}[section]
\newtheorem{Proposition}[Theorem]{Proposition}
\newtheorem{cor}[Theorem]{Corollary}
\newtheorem{Lemma}[Theorem]{Lemma}
\theoremstyle{remark}
\newtheorem{Definition}[Theorem]{Definition}
\newtheorem{Example}[Theorem]{Example}

\newtheorem{Remark}[Theorem]{Remark}

\begin{document}

\title{Strong pseudo-amenability of some Banach algebras}

\author[A. Sahami]{A. Sahami}

\email{amir.sahami@aut.ac.ir}

\address{Faculty of Basic sciences, Department of Mathematics, Ilam University, P.O.Box 69315-516, Ilam,
Iran.}



\maketitle

\begin{abstract}
In this paper we introduce a new notion of strong pseudo-amenability for Banach algebras. We study strong pseudo-amenability of some Matrix algebras. Using this tool, we characterize strong pseudo-amenability of $\ell^{1}(S)$, provided that $S$ is a uniformly locally finite semigroup. As an application we show that for a Brandt semigroup $S=M^{0}(G,I)$, $\ell^{1}(S)$ is strong pseudo-amenable if and only if $G$ is amenable and $I$ is finite. We give some examples to show the differences of strong pseudo-amenability and other classical notions of amenability.
\end{abstract}
\section{Introduction and Preliminaries}

Johnson introduced the class  of amenable Banach algebras. A Banach algebra $A$ is called amenable, if there exists a bounded net $(m_{\alpha})$ in $A\otimes_{p}A$ such that $a\cdot m_{\alpha}-m_{\alpha}\cdot a\rightarrow 0$ and $\pi_{A}(m_{\alpha})a\rightarrow a$ for every $a\in A.$ For further information about the history of amenability see \cite{run}.

By removing the boundedness condition in the definition of amenability, Ghahramani and Zhang in \cite{ghah pse} introduced and studied two generalized notions of amenability, named pseudo-amenability and pseudo-contractibility. A Banach algebra $A$ is called pseud-amenable(pseudo-contractible)  if there exists a not necessarily bounded net $(m_{\alpha})$ in  $A\otimes_{p}A$ such that $a\cdot m_{\alpha}-m_{\alpha}\cdot a\rightarrow 0(a\cdot m_{\alpha}=m_{\alpha}\cdot a)$ and $\pi_{A}(m_{\alpha})a\rightarrow a$ for every $a\in A,$ respectively.  Recently pseudo-amenablity and pseudo-contractiblity of the  archimedean semigroup algebras  and the uniformly locally finite semigroup algebras have  investigated in \cite{rost}, \cite{rost1} and \cite{maede}. 
In fact the main results of   \cite{rost} and  \cite{rost1} are about characterizing pseudo-amenability and pseudo-contractibility of  $\ell^{1}(S)$, where $S=M^{0}(G,I)$ is  the Brandt semigroup over an index set $I$. They showed that $\ell^{1}(S)$ is pseudo-amenable (pseudo-contractible) if and only if $G$ is amenable ($G$ is a finite group and $I$ is a finite index set), respectively.
 For further information about pseudo-amenability and pseudo-contractibility of general Banach algebras the readers refer to \cite{Choi}.

Motivated by these considerations this question raised  "Is there a notion of amenability which stands between pseudo-amenability and pseudo-contarctibility for the Brandt semigroup algebras?", that is, under which  notion of amenability for  the Brand semigroup algebra $\ell^{1}(M^{0}(G,I))$,    $G$  becomes  amenable and $I$  becomes finite. In order to answer this question author defines a new notion of amenability, named strong pseudo-amenability.  Here we give the definition of our new notion.
\begin{Definition}
A Banach algebra $A$ is called {\it strong pseudo-amenable}, if there exists a (not necessarily bounded) net $(m_{\alpha})_{\alpha}$ in $(A\otimes_{p}A)^{**}$ such that $$a\cdot m_{\alpha}-m_{\alpha}\cdot a\rightarrow 0,\quad a\pi_{A}^{**}(m_{\alpha})=\pi_{A}^{**}(m_{\alpha})a\rightarrow a\qquad(a\in A).$$
\end{Definition}
In this paper, we study the basic properties of strong pseudo-amenable Banach algebras. We show that strong pseudo-amenability is weaker than pseudo-contractibility but it is stronger than pseudo-amenability. We investigate strong pseudo-amenability of matrix algebras. Using this tool we characterize strong pseudo-amenability of $\ell^{1}(S)$, whenever $S$ is a uniformly locally finite semigroup. In particular,  we show that $\ell^{1}(S)$ is strong pseudo amenable if and only if $I$ is a finite index set and $G$ is amenable, where $S=M^{0}(G,I)$ is the Brandt semigroup. Finally we give some examples that shows the differences between strong pseudo-amenability and other classical concepts of amenability.

We present  some standard notations and definitions that we
shall need in this paper. Let $A$ be a Banach algebra. If $X$ is a
Banach $A$-bimodule, then  $X^{*}$ is also a Banach $A$-bimodule via
the following actions
$$(a\cdot f)(x)=f(x\cdot a) ,\hspace{.25cm}(f\cdot a)(x)=f(a\cdot x ) \hspace{.5cm}(a\in A,x\in X,f\in X^{*}). $$
Let $A$ and $B$ be   Banach algebras. The projective tensor product
 $A\otimes_{p}B$ with the following multiplication is a Banach algebra
$$(a_{1}\otimes b_{1})(a_{2}\otimes b_{2})=a_{1}a_{2}\otimes b_{1}b_{2}\quad (a_{1},a_{2}\in A, b_{1}b_{2}\in B).$$
Also $A\otimes_{p}A$ with the following action becomes a Banach $A-$bimodule:
$$a_{1}\cdot a_{2}\otimes a_{3}=a_{1} a_{2}\otimes a_{3},\quad a_{2}\otimes a_{3}\cdot a_{1}=a_{2}\otimes a_{3}a_{1},\quad (a_{1},a_{2},a_{3}\in A).$$
 The product morphism
$\pi_{A}:A\otimes_{p}A\rightarrow A$ is  specified by
$\pi_{A}(a\otimes b)=ab$ for every $a,b\in A$.

\section{Basic properties of strong pseudo-amenability}
\begin{Proposition}\label{strong give psudo}
Let $A$ be a strong pseudo-amenable Banach algebra. Then $A$ is pseudo-amenable.
\end{Proposition}
\begin{proof}
Since  $A$  is strong pseudo-amenable, there exists a net $(m_{\alpha})$ in $(A\otimes_{p}A)^{**}$  such that $a\cdot m_{\alpha}-m_{\alpha}\cdot a\rightarrow 0$ and $\pi^{**}_{A}(m_{\alpha})a=a\pi^{**}_{A}(m_{\alpha})\rightarrow a,$ for every $a\in A.$ Take $\epsilon>0$ and arbitrary finite subsets $F\subseteq A, \Lambda\subseteq (A\otimes_{p}A)^{*}$ and $L\subseteq A^{*}.$ It follows that $$||a\cdot m_{\alpha}-m_{\alpha}\cdot a||<\epsilon,\quad ||\pi^{**}_{A}(m_{\alpha})a-a||<\epsilon,$$ for every $a\in F.$  It is well-known that for each $\alpha$, there exists a net $(n^{\alpha}_{\beta})$ in $A\otimes_{p}A$ such that $n^{\alpha}_{\beta}\xrightarrow{w^{*}}m_{\alpha}.$ Thus using  $w^*$-continuty of  $\pi^{**}_{A}$, we have
 $$\pi_{A}(n^{\alpha}_{\beta})=\pi^{**}_{A}(n^{\alpha}_{\beta})\xrightarrow{w^{*}}\pi_{A}^{**}(m_{\alpha}).$$
Hence there exists $\beta=\beta(\epsilon,F,\Lambda, L)$ such that $$|a\cdot n^{\alpha}_{\beta(\epsilon,F,\Lambda, L)}(f)-a\cdot m_{\alpha}(f)|<\frac{\epsilon}{K},\quad |n^{\alpha}_{\beta(\epsilon,F,\Lambda, L)}\cdot a(f)-m_{\alpha}\cdot a(f)|<\frac{\epsilon}{K}$$
and $$|\pi_{A}(n^{\alpha}_{\beta(\epsilon,F,\Lambda, L)})a(g)-\pi^{**}_{A}(m_{\alpha})a(g)|<\frac{\epsilon}{2L_{0}},\quad | \pi^{**}_{A}(m_{\alpha})a(g)-a(g)|<\frac{\epsilon}{2L_{0}} $$ for every $a\in F, f\in \Lambda$ and $g\in L,$ where $K=\sup\{||f||:f\in \Lambda\}$ and $L_{0}=\sup\{||f||:f\in L\}$. So for a $c\in \mathbb{R}^{+}$ we have 
$$|a\cdot n^{\alpha}_{\beta(\epsilon,F,\Lambda, L)}-n^{\alpha}_{\beta(\epsilon,F,\Lambda, L)}\cdot a|<c\frac{\epsilon}{K}$$
and $$|a(g)-\pi_{A}(n^{\alpha}_{\beta(\epsilon,F,\Lambda, L)})a(g)|<\frac{\epsilon}{L_{0}},$$ for every $a\in F, f\in \Lambda$ and $g\in L$.
It follows that there exists a  net $(n^{\alpha}_{\beta(\epsilon,F,\Lambda, L)})_{(\alpha,\epsilon,F,\Lambda, L)}$ in $A\otimes_{p}A$ which satisfies $$a\cdot n^{\alpha}_{\beta(\epsilon,F,\Lambda, L)}-n^{\alpha}_{\beta(\epsilon,F,\Lambda, L)}\cdot a\xrightarrow{w}0,\quad \pi_{A}(n^{\alpha}_{\beta(\epsilon,F,\Lambda, L)})a-a\xrightarrow{w}0,\qquad (a\in A).$$
Using Mazur Lemma we can assume that 
$$a\cdot n^{\alpha}_{\beta(\epsilon,F,\Lambda, L)}-n^{\alpha}_{\beta(\epsilon,F,\Lambda, L)}\cdot a\xrightarrow{||\cdot||}0,\quad \pi_{A}(n^{\alpha}_{\beta(\epsilon,F,\Lambda, L)})a-a\xrightarrow{||\cdot||}0,\qquad (a\in A).$$ Therefore $A$ is pseudo-amenable.
\end{proof}
 Let $A$ be a Banach algebra and $\phi\in\Delta(A).$
A Banach algebra $A$ is called approximately left $\phi-$amenable, if there exists a (not nessecarily bounded) net $(n_{\alpha})$ in $A$ such that $$am_{\alpha}-\phi(a)m_{\alpha}\rightarrow 0,\quad \phi(m_{\alpha})\rightarrow 1,\qquad (a\in A).$$
For further information see \cite{agha}.
\begin{cor}\label{ap phi}
Let $A$ be a Banach algebra and $\phi\in\Delta(A).$ If $A$ is strong pseudo-amenable, then $A$ is approximately left $\phi$-amenable.
\end{cor}
\begin{proof}
By Proposition \ref{strong give psudo} strong pseudo-amenability of $A$ implies that $A$ is pseudo-amenable. So there exists a net $(m_{\alpha})$ in $A\otimes_{p}A$ such that $a\cdot m_{\alpha}-m_{\alpha}\cdot a\rightarrow 0$ and $\pi_{A}(m_{\alpha})a\rightarrow a,$ for every $a\in A.$ Define $T:A\otimes_{p}A\rightarrow A$ by $T(a\otimes b)=\phi(b)a$ for every $a,b\in A.$ Clearly $T$ is a bounded linear map. Set 
$n_{\alpha}=T(m_{\alpha})$. One can easily see that $$an_{\alpha}-\phi(a)n_{\alpha}\rightarrow 0,\quad \phi(n_{\alpha})=\phi(T(m_{\alpha}))
=\phi(\pi_{A}(m_{\alpha}))\rightarrow 1,\qquad (a\in A).$$ Then $A$ is approximately left $\phi$-amenable.
\end{proof}
A Banach algebra $A$ is called pseudo-contractible if there exists a net $(m_{\alpha})$ in $A\otimes_{p}A$ such that $a\cdot m_{\alpha}=m_{\alpha}\cdot a$ and $\pi_{A}(m_{\alpha})a\rightarrow a $ for each $a\in A,$ see\cite{ghah pse}.
\begin{Lemma}\label{pseudo-con}
Let $A$ be a  pseudo-contractible Banach algebra. Then $A$ is strong pseudo-amenable.
\end{Lemma}
\begin{proof}
Clear.
\end{proof}
\begin{Lemma}
Let $A$ be a  commutative pseudo-amenable  Banach algebra. Then $A$ is strong pseudo-amenable.
\end{Lemma}
\begin{proof}
Clear.
\end{proof}
A Banach algebra $A$ is called biflat if there exists a bounded $A$-bimodule morphism $\rho:A\rightarrow (A\otimes_{p}A)^{**}$ such that $\pi_{A}^{**}\circ \rho(a)=a $ for each $a\in A.$ See \cite{run}.
\begin{Lemma}\label{biflat}
Let $A$ be a  biflat  Banach algebra with a central approximate identity. Then $A$ is strong pseudo-amenable.
\end{Lemma}
\begin{proof}
Since $A$ is biflat,  there exists a bounded $A$-bimodule morphism $\rho:A\rightarrow (A\otimes_{p}A)^{**}$ such that $\pi_{A}^{**}\circ \rho(a)=a $ for each $a\in A.$ Let $(e_{\alpha})$ be a central approximate identity for $A$. Define $m_{\alpha}=\rho(e_{\alpha})$. Since $\rho $ is a bounded $A$-bimodule morphism,  we have $a\cdot m_{\alpha}=m_{\alpha}\cdot a$ and $\pi^{**}_{A}(m_{\alpha})a=a\pi^{**}_{A}(m_{\alpha})\rightarrow a,$ for every $a\in A.$
\end{proof}
\begin{Remark}
In the previous lemma we can replace the biflatness with the existence of a (not necessarily bounded net) of $A$-bimodule morphism $\rho_{\alpha}:A\rightarrow (A\otimes_{p}A)^{**}$ which satisfies $\pi_{A}\circ\rho_{\alpha}(a)\xrightarrow{||\cdot||}a$. Now using the similar argument as in the proof of previous and iterated limit theorem \cite[p. 69]{kel}, we can see that $A$ is strong pseudo-amenable.
\end{Remark}
\begin{Proposition}\label{epi}
Suppose that $A$ and $B$ are Banach algebras. Let $A$ be strong pseudo-amenable.  If $T:A\rightarrow B$  is a continuous epimorphism, then $B$ is strong pseudo-amenable.
\end{Proposition}
\begin{proof}
Since $A$ is strong pseudo-amenable, there exists a net $(m_{\alpha})$ in $(A\otimes_{p}A)^{**}$  such that $a\cdot m_{\alpha}-m_{\alpha}\cdot a\rightarrow 0$ and $\pi^{**}_{A}(m_{\alpha})a=a\pi^{**}_{A}(m_{\alpha})\rightarrow a,$ for every $a\in A.$
Define $T\otimes T:A\otimes_{p}A\rightarrow B\otimes_{p}B$ by $T\otimes T(a\otimes b)=T(a)\otimes T(b)$ for every $a,b\in A.$ Clearly $T\otimes T$ is a bounded linear map.   So we have
$$T(a)\cdot(T\otimes T)^{**}(m_{\alpha})-(T\otimes
T)^{**}(m_{\alpha})\cdot T(a)=(T\otimes T)^{**}(a\cdot
m_{\alpha}-m_{\alpha}\cdot a)\rightarrow 0,\quad (a\in A).
$$
and 
\begin{equation}
\begin{split}
&\pi^{**}_{B}\circ (T\otimes
T)^{**}(m_{\alpha})T(a)-T(a)\pi^{**}_{B}\circ (T\otimes
T)^{**}(m_{\alpha})\\
&=(\pi_{B}\circ (T\otimes
T))^{**}(m_{\alpha}\cdot
a)-(\pi_{B}\circ (T\otimes
T))^{**}(a\cdot m_{\alpha})\\
&=T^{**}\circ\pi^{**}_{A}(m_{\alpha}\cdot a)-T^{**}\circ\pi^{**}_{A}(a\cdot m_{\alpha})\\
&=T^{**}(\pi^{**}_{A}(m_{\alpha})a-a\pi^{**}_{A}(m_{\alpha}))=T^{**}(0)=0
\end{split}
\end{equation}
Also $$\pi^{**}_{B}\circ (T\otimes
T)^{**}(m_{\alpha})T(a)-T(a)=(\pi_{B}\circ (T\otimes
T))^{**}(m_{\alpha}\cdot
a)-T(a)=T^{**}(\pi^{**}_{A}(m_{\alpha})a-a)\rightarrow 0,$$ for
every $a\in A.$ Then $B$ is strong pseudo-amenable.
\end{proof}
\begin{cor}
Let $A$ be a Banach algbera and $I$ be a closed ideal of $A.$ If $A$ is strong pseudo-amenable, then $\frac{A}{I}$ is strong pseudo-amenable.
\end{cor}
\begin{proof}
The quotient map is  a bounded epimorphism from $A$ onto
$\frac{A}{I}$, now apply   previous proposition.
\end{proof}
\begin{Lemma}\label{idempotent}
Let $A$ and $B$ be Banach algebras. Suppose that $B$ has a non-zero idempotent. If $A\otimes_{p}B$ is strong pseudo-amenable, then $A$ is strong pseudo-amenable.
\end{Lemma}
\begin{proof}
It deduces from a small modification of the argument of \cite[Proposition 3.5]{joh1}. In fact suppose that $b_{0}$ is a non-zero idempotent  of $B.$ Using Hahn-Banach theorem there exists a bounded linear map $f\in B^{*}$ such that $f(bb_{0})=f(b_{0}b)$ and $ f(b_{0})=1$, for every $b\in B.$ Define $T_{b_{0}}:A\otimes_{p}B\otimes_{p}A\otimes_{p}B\rightarrow A\otimes_{p}A$ by $T(a_{1}\otimes b_{1}\otimes a_{2}\otimes b_{2})=f(b_{0}b_{1}b_{2})a_{1}\otimes a_{2}$ for each $a_{1},a_{2}\in A$ and $b_{1}, b_{2}\in B$. Since $A\otimes_{p}B$ is strong pseudo-amenable, there exists a net $(m_{\alpha})$ in $(A\otimes_{p}B\otimes_{p}A\otimes_{p}B)^{**}$ such that $$x\cdot m_{\alpha}-m_{\alpha}\cdot x\rightarrow 0,\quad \pi^{**}_{A\otimes_{p}B}(m_{\alpha})x=x\pi^{**}_{A\otimes_{p}B}(m_{\alpha})\rightarrow x,\qquad (x\in A\otimes_{p}B).$$ Now one can readily see that 
$$a\cdot T^{**}_{b_{0}}(m_{\alpha})- T^{**}_{b_{0}}(m_{\alpha})a\rightarrow 0,\quad a\pi^{**}_{A} \circ T^{**}_{b_{0}}(m_{\alpha})=\pi^{**}_{A} \circ T^{**}_{b_{0}}(m_{\alpha})a\rightarrow a,$$
for each $a\in A.$ It follows that $A$ is strong pseudo-amenable.
\end{proof}
\section{Strong psudo-amenability of matrix algebras}
Let $A$ be a Banach algebra and $I$ be a totally ordered set. The
set of  $I\times I$ upper triangular matrices, with entries 
from $A$ and  the usual matrix
operations and also finite $\ell^1$-norm, is a Banach algebra and it  denotes with $UP(I,A)$.
\begin{Theorem}\label{main}
Let $I$ be a totally ordered set with smallest element and let $A$ be a Banach algebra with $\phi\in\Delta(A)$.  Then $UP(I,A)$ is strong pseudo-amenable if and
only if $A$  is strong pseudo-amenable and  $|I|=1$.
\end{Theorem}
\begin{proof}
Let $i_{0}$ be the smallest element and $\phi\in\Delta(A).$ Suppose that  $UP(I,A)$ is strong pseudo-amenable.  Suppose conversely that $|I|>1$. Define $\psi_{\phi}:UP(I,A)\rightarrow \mathbb{C}$ by $\psi((a_{i,j})_{i,j})=\phi(a_{i_{0},i_{0}})$ for every $(a_{i,j})_{i,j}\in UP(I,A)$. Clearly $\psi_{\phi}$ is a character on $ UP(I,A)$. Since $UP(I,A)$ is strong pseudo-amenable, by Corollary \ref{ap phi} $UP(I,A)$  is approximate left $\psi_{\phi}-$amenable. So by \cite{agha} there esists a net $(n_{\alpha})$ in $UP(I,A)$ such that $an_{\alpha}-\psi_{\phi}(a)n_{\alpha}\rightarrow 0$ and $\psi_{\phi}(n_{\alpha})\rightarrow 1$ for every $a\in UP(I,A)$. Set $$J=\{(a_{i,j})\in UP(I,A)|a_{i,j}=0,\quad i\neq i_{0} \}.$$
It is easy to see that $J$ is a closed ideal of $UP(I,A)$ and $\psi_{\phi}|_{J}\neq 0$. So there exists a $j$ in $J$ such that $\psi_{\phi}(j)=1.$ Replacing  $(n_{\alpha})$ with $(n_{\alpha}j)$ we can assume that  $(n_{\alpha})$ is a net in $J$ such that  $an_{\alpha}-\psi_{\phi}(a)n_{\alpha}\rightarrow 0$ and $\psi_{\phi}(n_{\alpha})\rightarrow 1$ for every $a\in J$. Suppose that  $n_{\alpha}$ in $J$ has a form 
$\left(\begin{array}{cccc} a^{\alpha}_{i_{0},i_{0}}&a^{\alpha}_{i_{0},i}&\cdots\\
0&0&\cdots\\
\colon&\cdots&\colon
\end{array}
\right),$ for some nets $( a^{\alpha}_{i_{0},i_{0}})$ and $ (a^{\alpha}_{i_{0},i})$ in $A$. Note that since $|I|>1,$ the matrix $n_{\alpha}$ must has at least two columns.  Also $\psi_{\phi}(n_{\alpha})\rightarrow 1$ implies that $\phi( a^{\alpha}_{i_{0},i_{0}})\rightarrow 1.$ Let $x$ be an element of $A$ such that $\phi(x)=1.$ Set $a=\left(\begin{array}{cccc}0&x&0&\cdots\\
0&0&0&\cdots\\
\colon&\colon&\cdots&\colon
\end{array}
\right)\in J$. Clearly $a\in \ker\psi_{\phi}$. Put $a$ in the following fact
$an_{\alpha}-\psi_{\phi}(a)n_{\alpha}\rightarrow 0$. It  follows that $ a^{\alpha}_{i_{0},i_{0}}x\rightarrow 0$. Application of $\phi $ on  $ a^{\alpha}_{i_{0},i_{0}}x\rightarrow 0$  implies that 
$\phi( a^{\alpha}_{i_{0},i_{0}}x)=\phi( a^{\alpha}_{i_{0},i_{0}})\phi(x)=\phi( a^{\alpha}_{i_{0},i_{0}})\rightarrow 0$, which is impossible. So $|I|=1$ and $UP(I,A)=A$ which implies that $A$ is  strong pseudo-amenable.

Converse is clear.

\end{proof}
 Suppose that $A$ is a Banach algebra and $I$ is a non-empty set. We denote   $M_{I}(A)$  for the Banach algebra of $I\times
 I$-matrices over $A$, with the finite $\ell^{1}$-norm and the
matrix multiplication. This class of Banach algebras belongs to $\ell^{1}$-Munn algebras, see \cite{essl}. 
We also  denote $\varepsilon_{i,j}$ for a matrix belongs to
$M_{I}(\mathbb{C})$ which $(i,j)$-entry is 1 and 0 elsewhere.
The map $\theta:M_{I}(A)\rightarrow A\otimes_{p}
M_{I}(\mathbb{C})$ defined by
$\theta((a_{i,j}))=\sum_{i,j}a_{i,j}\otimes \varepsilon_{i,j}$ is an
isometric algebra isomorphism.
\begin{Theorem}
Let  $I$ be a non-empty set. Then $M_{I}(\mathbb{C})^{**}$ is strong pseudo-amenable if and only if $I$ is finite.
\end{Theorem}
\begin{proof}
Let $A=M_{I}(\mathbb{C}).$ Suppose that $A^{**}$ is strong pseudo-amenable.  There exists a net  $(m_{\alpha})$ in $(A^{**}\otimes_{p}A^{**})^{**}$ such that $a\cdot m_{\alpha}-m_{\alpha}\cdot a\rightarrow 0$ and $\pi^{**}_{A^{**}}(m_{\alpha})a=a\pi^{**}_{A^{**}}(m_{\alpha})\rightarrow a,$ for every $a\in A^{**}.$  So there exists a net $n_{\alpha}=\pi^{**}_{A^{**}}(m_{\alpha})$ in $A^{****}$ such that  $n_{\alpha}a=an_{\alpha}\rightarrow a,$ for every $a\in A.$ Thus for each $\alpha$ we have a net $(m^{\alpha}_{\beta})_{\beta}$ in $A^{**}$ such that $m^{\alpha}_{\beta}\xrightarrow{w^*}n_{\alpha}$ and $||m^{\alpha}_{\beta}||\leq ||n_{\alpha}||$. Then we have 
$$am^{\alpha}_{\beta}-m^{\alpha}_{\beta}a\xrightarrow{w} 0,\quad m^{\alpha}_{\beta}(f)\rightarrow n_{\alpha}(f),$$
where $f\in A^{***}$ such that $n_{\alpha}(f)\neq 0.$ Take $\epsilon>0$ and $F=\{a_{1},a_{2},a_{3},...,a_{r}\}$ an arbitrary subset of $A$ . Define $$V_{\alpha}=\{(a_{1}n-na_{1},a_{2}n-na_{2},...,a_{r}n-na_{r}, n(f)-n_{\alpha}(f))|n\in A^{**},||n||\leq||n_{\alpha}||\},$$ clearly $V_{\alpha}$ is a convex subset of $(\prod^{r}_{i=1}A^{**})\oplus_{1}\mathbb{C}.$  It is easy to see that $(0,0,,...,0)$  belongs to $\overline{V_{\alpha}}^{w}$. Since the norm topology and the weak topology on the convex sets are the same, we can assume that  $(0,0,,...,0)$  belongs to $\overline{V_{\alpha}}^{||\cdot||}$. So there exists an element  $m_{(F,\epsilon)}$ in  $A^{**}$ which $$||a_{i}m_{(F,\epsilon)}-m_{(F,\epsilon)}a_{i}||<\epsilon,\quad | m_{(F,\epsilon)}(f)- n_{\alpha}(f)| <\epsilon,\qquad ||m_{(F,\epsilon)} ||\leq ||n_{\alpha}||,$$
for every $i\in \{1,2,,...,r\}$. It follows that the net $(m_{(F,\epsilon)})_{(F,\epsilon)}$ in $A^{**}$ satisfies 
$$am_{(F,\epsilon)}-m_{(F,\epsilon)}a\xrightarrow{||\cdot||}0,\quad  m_{(F,\epsilon)}(f)\xrightarrow{|\cdot|} n_{\alpha}(f),\qquad ||m_{(F,\epsilon)} ||\leq ||n_{\alpha}||,$$ for every $a\in A.$ 
Since $n_{\alpha}(f)\neq 0$ we may assume that  $m_{(F,\epsilon)}(f)$ stays away from $0$.
 On  the other hand there exists a net $(m^{(F,\epsilon)}_{v})$ in $A$ such that $m^{(F,\epsilon)}_{v}\xrightarrow{w^{*}}m_{(F,\epsilon)}$ and $||m^{(F,\epsilon)}_{v}||\leq ||m_{(F,\epsilon)}||\leq ||n_{\alpha}||.$ So $$am^{(F,\epsilon)}_{v}\xrightarrow{w^{*}}am_{(F,\epsilon)},\quad m^{(F,\epsilon)}_{v}a\xrightarrow{w^{*}}m_{(F,\epsilon)}a,\qquad (a\in A).$$ Since $am_{(F,\epsilon)}-m_{(F,\epsilon)}a\xrightarrow{||\cdot||}0$, we may assume that $am_{(F,\epsilon)}-m_{(F,\epsilon)}a\xrightarrow{w^{*}}0$. It follows that 
\begin{equation}
\begin{split}
&w^{*}-\lim_{(F,\epsilon)}w^{*}-\lim_{v}(am^{(F,\epsilon)}_{v}-m^{(F,\epsilon)}_{v}a)\\
&=w^{*}-\lim_{(F,\epsilon)}w^{*}-\lim_{v}(am^{(F,\epsilon)}_{v}-am_{(F,\epsilon)}+am_{(F,\epsilon)}-m_{(F,\epsilon)}a+m_{(F,\epsilon)}a -m^{(F,\epsilon)}_{v}a)=0
\end{split}
\end{equation}
and $$w^{*}-\lim_{(F,\epsilon)}w^{*}-\lim_{v}m_{v}^{(F,\epsilon)}=n_{\alpha}.$$ 
Now using iterated limit theorem \cite[p. 69]{kel}, we can find a net $(m_{(v,F,\epsilon)})$ in $A$ such that 
$$w^{*}-\lim_{(v,F,\epsilon)}am_{(v,F,\epsilon)}-m_{(v,F,\epsilon)}a=0,\quad w^{*}-\lim_{(v,F,\epsilon)}m_{(v,F,\epsilon)} =n_{\alpha},\quad ||m_{(v, F,\epsilon)}||\leq|| n_{\alpha}||\quad (a\in A).$$ Since  $(m_{(v,F,\epsilon)})$ is a net  in $A$, we have  $am_{(v,F,\epsilon)}-m_{(v,F,\epsilon)}a\xrightarrow{w}0$.  Now we follow the similar arguments as in \cite[Example 4.1(iii)]{sah new amen} to show that $I$ is finite. 

Let $m_{(v,F,\epsilon)}=(y^{i,j}_{(v,F,\epsilon)})$, where
$y^{i,j}_{(v,F,\epsilon)}\in \mathbb{C}$ for every $i,j\in I$. Since the product
of the weak topology on $\mathbb{C}$ coincides with the weak
topology on $A$ \cite[Theorem 4.3]{schae}, for a fixed
$i_{0}\in\Lambda$, we have  $\varepsilon_{i_{0},j}m_{(v,F,\epsilon)}-m_{(v,F,\epsilon)}\varepsilon_{i_{0},j}\xrightarrow{w}0$.
Thus 
$y^{j,j}_{(v,F,\epsilon)}-y^{i_{0},i_{0}}_{(v,F,\epsilon)}\xrightarrow{w}0$ and
$y^{i,j}_{(v,F,\epsilon)}\xrightarrow{w}0$, whenever $i\neq j$. The boundedness of 
$(m_{(v,F,\epsilon)})$, implies that  $(y^{i_{0},i_{0}}_{(v,F,\epsilon)})$ is a
bounded net in $\mathbb{C}$. Thus $(y^{i_{0},i_{0}}_{(v,F,\epsilon)})$ has a
convergence subnet,  denote it again with
$(y^{i_{0},i_{0}}_{(v,F,\epsilon)})$. Suppose that    $(y^{i_{0},i_{0}}_{(v,F,\epsilon)})$ converges to $l$ with
respect to $|\cdot|$. On the other hand 
$y^{j,j}_{(v,F,\epsilon)}-y^{i_{0},i_{0}}_{(v,F,\epsilon)}\xrightarrow{w}0$, implies that 
$y^{j,j}_{(v,F,\epsilon)}-y^{i_{0},i_{0}}_{(v,F,\epsilon)}\xrightarrow{|.|}0$ (because $\mathbb{C}$ is a Hilbert space). Thus 
 $y^{j,j}_{(v,F,\epsilon)}\xrightarrow{|.|}l$ for every  $j\in I.$
We claim that $l\neq 0$. On the contrary suppose that $l=0$. Then  by
\cite[Theorem 4.3]{schae} we have $m_{(v,F,\epsilon)}\xrightarrow{w}0$. Then
$f(m_{(v,F,\epsilon)})\rightarrow 0$. Also we have 
$f(m_{(v,F,\epsilon)})=m_{(v,F,\epsilon)}(f)\rightarrow n_{\alpha}(f)\neq 0$, which reveals a contradiction. Hence  $l$ must be a non-zero number. Therefore the facts 
$y^{j,j}_{(v,F,\epsilon)}-y^{i_{0},i_{0}}_{(v,F,\epsilon)}\xrightarrow{w}0$
and $y^{i,j}_{(v,F,\epsilon)}\xrightarrow{w}0$ injunction with \cite[Theorem 4.3]{schae}, give that $y_{(v,F,\epsilon)}\xrightarrow{w}y_{0}$, where
$y_{0}$ is a matrix with $l$ in the  diagonal position  and $0$
elsewhere. So we have  $y_{0}\in
\overline{\hbox{Conv}(y_{(v,F,\epsilon)})}^{w}=\overline{\hbox{Conv}(y_{(v,F,\epsilon)})}^{||.||}$.
It implies that  $y_{0}\in A$. But $\infty =\sum_{j\in I}|l|= \sum_{j\in
I}|y^{j,j}_{0}|=||y_{0}||<\infty,$ provided that $I$ is
infinite which is a contradiction. So $I$ must be finite.

For converse, let $I$ be finite. Then $M_{I}(\mathbb{C})^{**}=M_{I}(\mathbb{C}^{**})=M_{I}(\mathbb{C})$.  Using \cite[Proposition 2.7]{rams} we know that   $M_{I}(\mathbb{C})$ is biflat with an identity. So Lemma \ref{biflat} implies that $M_{I}(\mathbb{C})$ is strong pseudo-amenable. 
\end{proof}
We can use the similar arguments as in the  previous theorem and shows  the following result. 
\begin{Theorem}\label{main}
Let  $I$ be a non-empty set. Then $M_{I}(\mathbb{C})$ is strong pseudo-amenable if and only if $I$ is finite.
\end{Theorem}
\begin{Remark}
We give a pseudo-amenable Banach algebra which is not strong pseudo-amenable. 

Let $I$ be an infinite set.  Using  \cite[Proposition 2.7]{rams}, $M_{I}(\mathbb{C})$ is biflat. By \cite[Proposition 3.6]{rost},  $M_{I}(\mathbb{C})$  has an approximate identity. Then \cite[Proposition 3.5]{rost} implies that  $M_{I}(\mathbb{C})$ is pseudo-amenable. But by previous theorem $M_{I}(\mathbb{C})$ is not strong pseudo-amenable.
\end{Remark}
\section{Some applications for Banach algebras related to locally compact groups}
In this section we study strong pseudo-amenability of the  measure algebras, the group algebras and some semigroup algebras related to  locally compact groups.
\begin{Proposition}\label{group algebra}
Let $G$ be a locally compact group. Then $L^{1}(G)$ is strong pseudo-amenable if and only if $G$ is amenable.
\end{Proposition}
\begin{proof}
Suppose that  $L^{1}(G)$ is strong pseudo-amenable. Then by Proposition \ref{strong give psudo}, $ L^{1}(G)$ is pseudo-amenable. So by \cite[Proposition 4.1]{ghah pse}, $G$ is amenable.

For converse, let $G$ be  amenable. By Johnson theorem  $L^{1}(G)$ is amenable. Therefore there exists $M\in (L^{1}(G)\otimes_{p}L^{1}(G))^{**}$ such that $a\cdot M=M\cdot a$  and $\pi_{L^{1}(G)}^{**}(M)a=a\pi_{L^{1}(G)}^{**}(M)=a$ for every $a\in L^{1}(G)$. Then $L^{1}(G)$ is strong pseudo-amenable.
\end{proof}
\begin{Remark}\label{rem}
In fact in the proof of the previous proposition we showed that, if a Banach algebra $A$ is amenable, then $A$ is strong pseudo-amenable.
\end{Remark}
\begin{Proposition}
Let $G$ be a locally compact group. Then $M(G)$ is strong pseudo-amenable if and only if $G$ is discrete and amenable.
\end{Proposition}
\begin{proof}
Suppose that  $M(G)$ is strong pseudo-amenable. Then by Proposition \ref{strong give psudo}, $M(G)$ is pseudo-amenable. So by \cite[Proposition 4.2]{ghah pse}, $G$ is discrete and  amenable.

For converse, let $G$ discrete and   amenable. Then by the  main result of \cite{dale ghah hel},  $M(G)$  is amenable. Applying Remark \ref{rem}, implies that $M(G)$ is strong pseudo-amenable.
\end{proof}
\begin{Proposition}
Let $G$ be a locally compact group. Then ${L^{1}(G)}^{**}$ is strong pseudo-amenable if and only if $G$ is finite.
\end{Proposition}
\begin{proof}
Suppose that  ${L^{1}(G)}^{**}$ is strong pseudo-amenable. Then by Proposition \ref{strong give psudo} ${L^{1}(G)}^{**}$ is pseudo-amenable. So by \cite[Proposition 4.2]{ghah pse}  $G$ is finite.

For converse, let $G$ finite . Clearly  ${L^{1}(G)}^{**}$  is amenable. Applying Remark \ref{rem}, implies that $M(G)$ is strong pseudo-amenable.
\end{proof}

We present  some notions of semigroup theory, Our standard reference of semigroup theory is  \cite{how}. Let $S$ be a semigroup and let $E(S)$ be
the set of its idempotents. There exists a  partial order on $E(S)$ which is defined by
$$s\leq t\Longleftrightarrow s=st=ts\quad (s,t\in E(S)).$$ A  semigroup $S$ is called inverse semigroup, if for every $s\in S$
there exists $s^{*}\in S$ such that $ss^{*}s=s^{*}$ and
$s^{*}ss^{*}=s$. If $S$ is an inverse semigroup,
then there exists a partial order on $S$ which  coincides with the
partial order on $E(S)$. Indeed
$$s\leq t\Longleftrightarrow s=ss^{*}t\quad (s,t\in
S).$$ For every  $x\in S$, we denote $(x]=\{y\in S|\,y\leq x\}$. $S$
is called locally finite (uniformly locally finite) if for each
$x\in S$, $|(x]|<\infty\,\,(\sup\{|(x]|\,:\,x\in S\}<\infty)$,
respectively.

Suppose that $S$ is an inverse semigroup. Then  the maximal subgroup
of $S$ at $p\in E(S)$ is denoted by $G_{p}=\{s\in
S|ss^{*}=s^{*}s=p\}$.

 Let $S$ be an inverse semigroup. There exists an equivalence relation $\mathfrak{D}$ on $S$
such that $s\mathfrak{D}t$ if and only if there exists $x\in S$ such
that $ss^{*}=xx^{*}$ and $t^{*}t=x^{*}x$. We denote
$\{\mathfrak{D}_{\lambda}:\lambda\in \Lambda\}$ for the collection
of $\mathfrak{D}$-classes and $E(\mathfrak{D}_{\lambda})=E(S)\cap
\mathfrak{D}_{\lambda}.$
\begin{Theorem}\label{inverse}
Let $S$ be an inverse semigroup such that $E(S)$ is uniformly
locally finite. Then the following are equivalent:
\begin{enumerate}
\item [(i)] $\ell^{1}(S)$ is strong pseudo-amenable;
\item [(ii)] Each maximal subgroup of $S$ is amenable and each $\mathfrak{D}$-class has finitely many
idempotent elements.
\end{enumerate}
\end{Theorem}
\begin{proof}
Let $\ell^{1}(S)$ be strong pseudo-amenable. Since $S$
is a uniformly locally finite inverse semigroup, using \cite[Theorem 2.18]{rams} we have
 $$\ell^{1}(S)\cong
\ell^{1}-\bigoplus\{M_{E(\mathfrak{D}_{\lambda})}(\ell^{1}(G_{p_{\lambda}}))\}.$$ Thus $M_{E(\mathfrak{D}_{\lambda})}(\ell^{1}(G_{p_{\lambda}}))$ is a homomorphic image of  $\ell^{1}(S).$ Then by  Proposition \ref{epi} strong pseudo-amenability of $\ell^{1}(S)$ implies that $M_{E(\mathfrak{D}_{\lambda})}(\ell^{1}(G_{p_{\lambda}}))$  is strong pseudo-amenable. It is well-known that  $\ell^{1}(G_{p_{\lambda}})$ has an identity (then has an idempotent element). Hence by Lemma \ref{idempotent}, $M_{E(\mathfrak{D}_{\lambda})}$ is strong pseudo-amenable. Now by Theorem \ref{main}, $E(\mathfrak{D}_{\lambda})$ is finite. Also since   $M_{E(\mathfrak{D}_{\lambda})}$ has an idempotent, again by Lemma \ref{idempotent}, $\ell^{1}(G_{p_{\lambda}})$ is strong pseudo-amenable. Applying Proposition \ref{group algebra}, $G_{p_{\lambda}}$ is amenable.

For converse, suppose that   $E(\mathfrak{D}_{\lambda})$ is finite
and  $G_{p_{\lambda}}$ is amenable for every $\lambda$.  Johnson
theorem implies that $\ell^{1}(G_{p_{\lambda}})$ is $1$-amenable (so it is
$1$-biflat).  By  \cite[Proposition 2.7]{rams} it follows that 
$M_{E(\mathfrak{D}_{\lambda})}(\mathbb{C})$ is $1$-biflat.
So   \cite[Proposition 2.5]{rams} gives  that
$M_{E(\mathfrak{D}_{\lambda})}(\mathbb{C})\otimes_{p}\ell^{1}(G_{p_{\lambda}})$
is $1$-biflat. Using $$\ell^{1}(S)\cong
\ell^{1}-\bigoplus\{M_{E(\mathfrak{D}_{\lambda})}(\ell^{1}(G_{p_{\lambda}}))\},$$
and \cite[Proposition 2.3]{rams}, we have $\ell^{1}(S)$ is $1$-biflat. The finiteness of 
 $E(\mathfrak{D}_{\lambda})$ deduces that 
$M_{E(\mathfrak{D}_{\lambda})}(\mathbb{C})\otimes_{p}\ell^{1}(G_{p_{\lambda}})$
has an identity.  Therefore  It is easy to see that  $\ell^{1}(S)$ has a central
approximate identity. Now by Lemma\ref{biflat},  biflatness of $\ell^{1}(S)$ gives that $\ell^{1}(S)$ is strong pseudo-amenable.
\end{proof}
We recall that a Banach algebra $A$ is approximately amenable, if
for each Banach $A$-bimodule $X$ and each bounded derivation
$D:A\rightarrow X^{*}$ there exists a net $(x_{\alpha})$ in $X^{*}$
such that $$D(a)=\lim_{\alpha}a\cdot x_{\alpha}-x_{\alpha}\cdot
a\quad (a\in A),$$ for more  details  see \cite{gen 1} and \cite{gen 2}.

For a locally compact group $G$ and a non-empty set $I$, set
$$M^{0}(G,I)=\{(g)_{i,j}:g\in G,i,j\in I\}\cup \{0\},$$
where $(g)_{i,j}$ denotes the $I\times I$ matrix with  $g$ in
$(i,j)$-position and zero elsewhere. With the following
multiplication $M^{0}(G,I)$ becomes a semigroup
\begin{eqnarray*}�
&\textit{�$(g)_{i,j}\ast
(h)_{k,l}=$}\begin{cases}(gh)_{il}\,\,\,\,\,\,\,\,\,\,\,\,\,\,\,\,\,\,
j=k\cr � 0\,\,\,\,\,\,\,\,\,\,\,\,\,\,\,\,\,\,     \qquad j\neq k,
�\end{cases}\\�
\end{eqnarray*}�
It is well known that $M^{0}(G,I)$ is an inverse semigroup with
$(g)^{*}_{i,j}=(g^{-1})_{j,i}$. This semigroup is called Brandt
semigroup over $G$ with index set $I,$ which by the arguments as in \cite[Corollary 3.8]{rost}, $M^{0}(G,I)$  becomes a uniformly locally finite inverse semigroup.

\begin{Theorem}\label{brandt}
Let $S=M^{0}(G,I)$ be a Brandt semigroup.   Then the following
are equivalent:
\begin{enumerate}
\item [(i)] $\ell^{1}(S)$ is strong pseudo-amenable;
\item [(ii)] $G$ is amenable and $I$ is finite;
\item [(ii)] $\ell^{1}(S)$ is approximately amenable.
\end{enumerate}
\end{Theorem}
\begin{proof}(i)$\Rightarrow$(ii) Using  \cite[Remark, p 315]{dun-nam}, we know that $\ell^{1}(S)$
is isometrically isomorphic with
$[M_{I}(\mathbb{C})\otimes_{p}\ell^{1}(G)]\oplus_{1}\mathbb{C}$. Applying Proposition \ref{epi}, $M_{I}(\mathbb{C})\otimes_{p}\ell^{1}(G)$ is strong  pseudo-amenable. Since  $\ell^{1}(G)$ has an identity, by Lemma \ref {idempotent} $M_{I}(\mathbb{C})$ is strong pseudo-amenable. Hence by Theorem \ref{main}, $I$ must be finite. On the other hand the  finiteness of $I$ implies that  $M_{I}(\mathbb{C})$ has a unit. So  Lemma \ref{idempotent}  implies that $\ell^{1}(G)$ is strong amenable. Now by Proposition \ref{group algebra} $G$ is amenable.

(ii)$\Rightarrow$(i) Similar to the proof of $(ii)\Rightarrow (i)$ of 
previous theorem.

(ii)$\Leftrightarrow$(iii) By the main result of \cite{sadr}, it is
clear.
\end{proof}
\begin{Remark}
There exists a pseudo-amenable  semigroup algebra  which is not strong pseudo-amenable. 

To see this, let $G$ be an  amenable locally compact group. Suppose that $I$ is an infinite set. By \cite[Corollary 3.8]{rost} $\ell^{1}(S)$ is pseudo-amenable but using previous theorem $\ell^{1}(S)$  is not strong pseudo-amenable, whenever $S=M^{0}(G,I)$ is a Brandt semigroup.

Also  there exists a strong  pseudo-amenable  semigroup algebra  which is not pseudo-contractible. 

To see this,  let $G$ be an infinite  amenable group. Suppose that $I$ is a finite set. By previous theorem $\ell^{1}(S)$ is strong pseudo-amenable but \cite[Corollary 2.5]{rost1} implies that $\ell^{1}(S)$  is not pseudo-contractible, whenever  $S=M^{0}(G,I)$ is  a Brandt semigroup.

\end{Remark}
\section{Examples}
\begin{Example}\label{ex1}
We present some strong pseudo-amenable Banach algebras which is not amenable.  
\begin{enumerate}
\item [(i)]  Suppose that $G$ is the integer Heisenberg group. We know that $G$ is discrete and amenable, see \cite{run}.  Therefore by the main result of \cite{for-run} the Fourier algebra $A(G)$ is not amenable. But by the  Leptin theorem (see\cite{run}), the  amenability of $G$ implies that $A(G) $  has a bounded approximate identity. Since $A(G)$ is a commutative Banach algebra,  $A(G)$  has a central approximate identity. Hence \cite[Theorem 4.2]{sam} follows that $A(G)$ is pseudo-contractible. Now  by Lemma \ref{pseudo-con}  $A(G)$ is strong pseudo-amenable.

\item [(ii)] Let $S=\mathbb{N}$. Equip $S$ with ${\it max}$ as its product. Then the semigroup algebra $\ell^{1}(S)$ is not amenable. To see this on contrary suppose that   $\ell^{1}(S)$ is amenable. Then by \cite{dun-nam},  $E(S)$ must be finite which is impposible. We claim that  $\ell^{1}(S)$ is strong pseudo-amenable. By \cite[p. 113]{dale lau struss} $\ell^{1}(S)$ is approximate amenable. Since  $\ell^{1}(S)$ has an identity,    $\ell^{1}(S)$ is pseudo-amenable. So commutativity of  $\ell^{1}(S)$ follows that  $\ell^{1}(S)$ is strong pseudo-amenable. 

\item [(iii)] Let $S=\mathbb{N}\cup\{0\}$.  With the following action
\begin{eqnarray*}�
&\textit{$m\ast
n=$}\begin{cases}m\,\,\,\,\,\,\,\,\,\,\,\,\,\,\,\,\,\ if\qquad
m=n\cr � 0\,\,\,\,\,\,\,\,\,\,\,\,\,\,\,\,\,\,\,\,if\qquad m\neq n,
�\end{cases}\\�
\end{eqnarray*}�   $S$ becomes a
semigroup. Clearly $S$ is commutative and $E(S)=\mathbb{N}\cup\{0\}.$  So by\cite{dun-nam}, $\ell^{1}(S)$  is not amenable. On the other hand since $S$ is a uniformly locally finite semilattice, \cite[Corollary 2.7]{rost1} implies that $\ell^{1}(S)$  is pseudo-contractible. Thus by Lemma \ref{pseudo-con}  $\ell^{1}(S)$ is strong pseudo-amenable. 
\end{enumerate}
\end{Example}

\begin{Example}
We give a strong pseudo-amenable Banach algebra which is not approximately amenable. 
 
A Banach algebra $A$ is 
approximately biprojective  if there exists a (not nessecarily bounded ) net $(\rho_{\alpha})$
of bounded linear $A$-bimodule morphisms from $A$ into
$A\otimes_{p}A$ such that $\pi_{A}\circ\rho_{\alpha}(a)-a\rightarrow
0,$ for every $a\in A,$ see \cite{zhang}. Suppose that
$A=\ell^{2}(\mathbb{N})$. With the pointwise multiplication, $A$
becomes a Banach algebra. By the main result of \cite{dale zhang},
$A$ is not approximately amenable. But by \cite[Example
p-3239]{zhang}, $A$ is an approximately biprojective Banach algebra
with a central approximate identity. Then by \cite[Proposition
3.8]{ghah pse}, $A$ is pseudo-contractible. It follows that  $A$ is strong pseudo-amenable.
\end{Example}
\begin{Example}
We give a biflat semigroup algebra which is not strong pseudo-amenable. So we can not remove the hypothesis " the existence of central approximate identity" from  Lemma \ref{biflat}.

Let $S$ be the right-zero semigroup with $|S|>1$, that is, $st=t$ for every $s,t\in S$. We denote $\phi_{S}$ for the augmentation character on $\ell^{1}(S)$. It is easy to show that $f g=\phi_{S}(f)g$. Pick $f_{0}\in \ell^{1}(S)$ such that $\phi_{S}(f_{0})=1$. Define $\rho:\ell^{1}(S)\rightarrow \ell^{1}(S)\otimes_{p}\ell^{1}(S)$ by $\rho(f)=f_{0}\otimes f$. It is easy to see that $\pi_{\ell^{1}(S)}\circ\rho(f)=f$ and $\rho$ is a bounded $\ell^{1}(S)$-bimodule morphism. So $\ell^{1}(S)$ is biflat. Suppose conversely that $\ell^{1}(S)$ is strong pseudo-amenable. So by Proposition \ref{strong give psudo}  $\ell^{1}(S)$ is pseudo-amenable. Hence $\ell^{1}(S)$ has an approximate identity, say $(e_{\alpha})$. It leads that $$f_{0}=\lim f_{0}e_{\alpha}=\lim \phi_{S}(f_{0})e_{\alpha}=\lim e_{\alpha}.$$ Suppose that $s_{1}, s_{2}$ be two arbitrary elements  in $S$. Thus $\delta_{s_{1}}=\lim \delta_{s_{1}} e_{\alpha}=\delta_{s_{1}} f_{0}=f_{0}$ and  $\delta_{s_{2}}=\lim \delta_{s_{2}} e_{\alpha}=\delta_{s_{2}} f_{0}=f_{0}$ which implies that $\delta_{s_{1}}=\delta_{s_{2}}$. Then $s_{1}=s_{2}$. Therefore $|S|=1.$ which is a contradiction.
\end{Example}
\begin{Remark}
A Banach algebra $A$ is called Johnson pseudo-contractible, if there exists a net $(m_{\alpha})_{\alpha}$ in $(A\otimes_{p}A)^{**}$ such that $$a\cdot m_{\alpha}=m_{\alpha}\cdot a, \quad \pi^{**}_{A}(m_{\alpha})a\rightarrow a,\qquad (a\in A). $$ For more information about Johnson pseudo-contractibility, see \cite{sah new amen}. By Example \ref{ex1}[ii] we know that $\ell^{1}(\mathbb{N}_{max})$ is  strong pseudo-amenable. But by \cite[Example 2.6]{askar}  $\ell^{1}(\mathbb{N}_{max})$ is not Johnson pseudo-contractible.
\end{Remark}
\begin{small}

\end{small}

\end{document}